\documentclass[twoside,12pt]{article}
\usepackage{graphicx,amsmath,latexsym,amssymb,amsthm}
\usepackage[colorlinks=black,urlcolor=black]{hyperref}
at 9pt  

\date{}

\newtheorem{thm}{Theorem}[section]
\newtheorem{lem}[thm]{Lemma}
\newtheorem{cor}[thm]{Corollary}

\newtheorem{rem}[thm]{Remark}

\numberwithin{equation}{section}

\begin{document}
\setlength{\unitlength}{1cm}



\vskip1.5cm

\centerline { \textbf{ Expansion Theorem for Sturm-Liouville problems
with }}
 \centerline { \textbf{transmission conditions  }}

\vskip.2cm


\vskip.8cm \centerline {\textbf{O. Sh. Mukhtarov$^\dag$ and K.
Aydemir$^\dag$ }}

\vskip.5cm

\centerline {$^\dag$Department of Mathematics, Faculty of Science,}
\centerline {Gaziosmanpa\c{s}a University,
 60250 Tokat, Turkey}
\centerline {e-mail : {\tt omukhtarov@yahoo.com,
kadriye.aydemir@gop.edu.tr }}




\vskip.5cm \hskip-.5cm{\small{\bf Abstract :} The purpose of this
paper is to extend some  spectral properties of regular
Sturm-Liouville problems to the special type discontinuous
boundary-value problem, which consists of a Sturm-Liouville equation
together with eigenparameter-dependent boundary conditions and two
supplementary transmission conditions. We construct  the resolvent
operator  and Green' s function and prove theorems about expansions
in terms of eigenfunctions in modified Hilbert space $L_{2}[a, b]$.
\vskip0.3cm\noindent {\bf Keywords :} \ Boundary-value problems,
transmission conditions, Resolvent operator, expansion theorem.
{\vskip0.3cm\noindent {\bf AMS subject classifications : 34L10,
34L15 }

\hrulefill

\section{\textbf{Introduction}}

With historical roots in the application of Fourier series to heat
flow, the Sturmian theory is one of the most extensively developing
fields in pure and applied mathematics. As is well-known  the
eigenvalue parameter takes part linearly only in the differential
equation in the classical Sturm-Liouville problems. However, in
mathematical physics are encountered such problems, where eigenvalue
parameter appear in both differential equation  and boundary
conditions. The first, we cite the works of Walter \cite{wa} and
Fulton \cite{fu} both of which have extensive bibliographies, in the
case of \cite{fu}, a discussion of physical applications.
Afterwards, we mention the results  \cite{bi1, bi2,
 ko, sh}  and corresponding references cited
therein. In recent years there has been increasing interest of some
Sturm-Liouville type problems which may have discontinuities in the
solution or its derivative at interior point (see \cite{be, ba, bo,
ch,  ka1, os1,  ka2, e.t, wang}).

In this paper we shall investigate some spectral properties of one
discontinuous Sturm-Liouville problem for which the eigenvalue
parameter takes part in both differential equation and boundary
conditions. Moreover, two supplementary transmission conditions at
one interior point are added to boundary conditions. Namely, we
consider the Sturm-Liouville equation,
\begin{equation}\label{1.1}
\tau u:=-u^{\prime \prime }(x)+ q(x)u(x)=\lambda u(x)
\end{equation}
to hold in finite interval $(a, b)$ except at one inner point $c \in
(a, b) $, where discontinuity in $u  \ \textrm{and} \ u'$   are
prescribed by transmission conditions
\begin{equation}\label{1.2}
\gamma_{1}u(c-0)-\delta_{1} u(c+0)=0,
\end{equation}
\begin{equation}\label{1.3}
\gamma_{2}u'(c-0)-\delta_{2} u('c+0)=0,
\end{equation}
together with the eigenparameter- dependent boundary conditions
\begin{equation}\label{1.4}
 \alpha_{1}u(a)+\alpha_{2}u'(a)=0,
\end{equation}
\begin{equation}\label{1.5}
(\beta'_{1}\lambda+\beta_{1})u(b)-(\beta'_{2}\lambda+\beta_{2})u'(b)=0,
\end{equation}
where the potential  $q(x)$ is real-valued, continuous in each
interval $[a, c) \  \textrm{and} \\ (c, b]$ and has a finite limits
$q( c\mp0)$ ; \ $\alpha_{i}, \ \beta_{i}, \ \beta'_{i}, \
\delta_{i}, \ \gamma_{i} \ (i=1,2)$ are real numbers; \ $\lambda$ \
is a complex eigenparameter. Naturally we exclude each of the
trivial conditions $\gamma_{1}=\delta_{1}=0, \
\gamma_{2}=\delta_{2}=0, \ \alpha_{1}= \alpha_{2}=0,  \
\beta'_{1}=\beta_{1}= \beta'_{2}=\beta_{2}=0 $. In contrast to
previous works, eigenfunctions of this problem may have
discontinuity at the one inner point of the considered interval.

This kind of problems are connected with discontinuous material
properties, such as heat and mass transfer, varied assortment of
physical transfer problems, vibrating string problems when the
string loaded additionally with point masses, and diffraction
problems \cite{ka2, tit}. The study of the structure of the solution
of the matching region leads to the consideration of an eigenvalue
problem for a second order differential operator with piecewise
continuous coefficients and transmission conditions at interior
points.  A. Boumenir  \cite{bo} use sampling techniques to
reconstruct the characteristic function associated with the
eigenvalues of two linked Sturm–Liouville operators by a
transmission condition. In \cite{wang}, Wang et al. studied a class
of Sturm-Liouville problems with eigenparameter-dependent boundary
conditions and transmission conditions at an interior point.  B.
Chanane \cite{ch} computed the eigenvalues of Sturm–Liouville
problems with several discontinuity conditions inside a finite
interval and parameter dependent boundary conditions using the
regularized sampling method. In \cite{ba} E. Bairamov and E.
U\u{g}urlu examined the determinants of dissipative Sturm-Liouville
operators with transmission conditions. J. Ao et al. \cite{be} have
considered the finite spectrum of Sturm–Liouville problems with
transmission conditions. Such properties as isomorphism,
coerciveness with respect to the spectral parameter, completeness of
root functions, distributions of eigenvalues of some discontinuous
boundary value problems with transmission conditions and its
applications to the corresponding initial-boundary value problems
for parabolic equations have been investigated in \cite{ka1, os1,
ka2} and \cite{tit}.
\section{The fundamental solutions and \\ Green' s function }
By following the procedure of \cite{os1} we can define four solution
$\phi_{1}(x,\lambda), \ \phi_{2}(x,\lambda), \\
\chi_{1}(x,\lambda)\ \textrm{and} \  \chi_{2}(x,\lambda)$ of the
equation (\ref{1.1}) under the initial conditions
\begin{equation}\label{3}
 \ \ u(a)=\alpha_{2} , \
\ u'(a)=-\alpha_{1},
\end{equation}
\begin{equation}\label{5}
 \ \ u(c+0)=\frac{\gamma_{1}}{\delta_{1}}\phi_{1}(c-0,\lambda) , \
\
u'(c+0)=\frac{\gamma_{2}}{\delta_{2}}\frac{\partial\phi_{1}(c-0,\lambda)}{\partial
x}
\end{equation}
\begin{equation}\label{4}
 \ \ u(b)=\beta'_{2}\lambda+\beta_{2} , \
\ u'(b)=\beta'_{1}\lambda+\beta_{1},
\end{equation}
and
\begin{equation}\label{6}
 \ \ u(c-0)=\frac{\delta_{1}}{\gamma_{1}}\chi_{2}(c+0,\lambda), \
\
u'(c-0)=\frac{\delta_{2}}{\gamma_{2}}\frac{\partial\chi_{2}(c+0,\lambda)}{\partial
x}
\end{equation}
respectively. Consequently, each of the functions
\begin{displaymath} \phi(x,\lambda)=\left
\{\begin{array}{ll}
\phi_{1}(x,\lambda)\  \textrm{for} \ x\in [a,c) \\
\phi_{2}(x,\lambda)\ \textrm{for} \ x\in (c,b] \\
\end{array}\right.
\ \chi(x,\lambda)=\left \{\begin{array}{ll}
\chi_{1}(x,\lambda) \ \textrm{for} \ x\in [a,c) \\
\chi_{2}(x,\lambda)\ \textrm{for} \ x\in (c,b]\\
\end{array}\right.
\end{displaymath}
satisfies the equation (\ref{1.1}) and the both transmission
conditions (\ref{1.2}) and (\ref{1.3}). Moreover, the solution
$\phi(x,\lambda)$ satisfies the first of boundary condition
(\ref{1.4}), but $\chi(x,\lambda)$  satisfies the other boundary
condition (\ref{1.5}). By applying the same method as in \cite{e.t}
we can prove that the solutions $\phi(x,\lambda) \ \textrm{and} \
\chi(x,\lambda)$ are entire functions of  complex parameter
$\lambda$ for each fixed $x \in [a,c)\cup (c,b]$.\\ It is known from
ordinary linear differential  equation  theory that each of the
Wronskians
$w_{1}(\lambda)=W(\phi_{1}(x,\lambda),\chi_{1}(x,\lambda)) \
\textrm{and}  \
w_{2}(\lambda)=W(\phi_{2}(x,\lambda),\chi_{2}(x,\lambda))$ are
independent of $x$ in $[a,c) \ \textrm{and} \ (c,b]$ respectively.
By using (\ref{5}) and (\ref{6}) we have
\begin{eqnarray} \label{7} w_{1}(\lambda) &=& \phi_{1}(c-0,\lambda)\frac{\partial
\chi_{1}(c-0,\lambda)}{\partial
x}-\chi_{1}(c-0,\lambda)\frac{\partial
\phi_{1}(c-0,\lambda)}{\partial x} \nonumber \\&=& \frac{\delta_{1}
\delta_{2}}{\gamma_{1}\gamma_{2}}(\phi_{2}(c+0,\lambda)
\frac{\partial \chi_{2}(c+0,\lambda)}{\partial
x}-\chi_{2}(c+0,\lambda)\frac{\partial
\phi_{2}(c+0,\lambda)}{\partial x})\nonumber
\\&=& \frac{\delta_{1}
\delta_{2}}{\gamma_{1}\gamma_{2}}w_{2}(\lambda)
\end{eqnarray}
Denote
 \begin{equation}\label{8}w(\lambda):=\gamma_{1}\gamma_{2} w_{1}(\lambda) = \delta_{1} \delta_{2} \
w_{2}(\lambda).\end{equation}

Again, similarly to \cite{os1} it can be proven that, there are
infinitely many eigenvalues $\lambda_{n}, \ n=1,2,...$   of the BVTP
$(\ref{1.1})-(\ref{1.5})$ which are coincide with
the zeros of characteristic function  \ $w(\lambda)$.\\
Let us consider the nonhomegeneous  differential equation
\begin{eqnarray}\label{2.5}
u''+(\lambda-q(x))u=F_{1}(x), \
\end{eqnarray}
on $[a,c)\cup(c,b]$ subject to nonhomogeneous boundary conditions
\begin{equation}\label{2.6}
 \alpha_{1}u(a)+\alpha_{2}u'(a)=0,
\end{equation}
\begin{equation}\label{2.7}
(\beta_{1}u(b)-\beta_{2}u'(b))+\lambda(\beta'_{1}u(b)-\beta'_{2}u'(b))=F_{2}
\end{equation}
and homogeneous transmission conditions
\begin{equation}\label{2.8}
\gamma_{1}u(c-0)-\delta_{1} u(c+0)=0, \ \
\end{equation}
\begin{equation}\label{2.89}
\gamma_{2}u'(c-0)-\delta_{2} u'(c+0)=0.
\end{equation}
and let $\lambda$ is not eigenvalue. Making use of the definitions
of the functions $ \phi_{i}, \chi_{i}$ (i = 1, 2) we find that the
general solution of the differential equation (\ref{2.5}) can be
written in the form
\begin{eqnarray}\label{2.10}
u(x,\lambda)=\left\{\begin{array}{c}
               \frac{\chi_{1}(x,\lambda)}{\omega_{1}(\lambda)}\int_{a}^{x}\phi_{1}(y,\lambda)F_{1}(y)dy +
\frac{\phi_{1}(x,\lambda)}{\omega_{1}(\lambda)}\int_{x}^{c}\chi_{1}(y,\lambda)F_{1}(y)dy \\ +c_{11}\phi_{1}(x,\lambda)+c_{12}\chi_{1}(x,\lambda) \ , \ \ \ \ \ \ \ \ for \  x \in (a,c) \\
                \\
               \frac{\chi_{2}(x,\lambda)}{\omega_{2}(\lambda)}\int_{c}^{x}\phi_{2}(y,\lambda)F_{1}(y)dy +
\frac{\phi_{2}(x,\lambda)}{\omega_{2}(\lambda)}\int_{x}^{b}\chi_{2}(y,\lambda)F_{1}(y)dy\\ +c_{21}\phi_{2}(x,\lambda)+c_{22}\chi_{2}(x,\lambda) \ , \ \ \ \ \ \ \ \ for \ x \in (c,b) \\
             \end{array}\right.
\end{eqnarray}
where $C_{ij} \ (i,j=1,2)$  are arbitrary constants. By substitution
into the boundary conditions (\ref{2.6}) and (\ref{2.7}) we see at
once that
\begin{eqnarray}\label{2.11}c_{12}=0, \
C_{21}=\frac{F_{2}}{\omega_{2}(\lambda)}.
\end{eqnarray}
Further, substitution (\ref{2.10}) into transmission conditions
(\ref{2.8}) and (\ref{2.89})we have the inhomogeneous linear system
of equations for $c_{11}$ \textrm{and} $c_{22}$  , the determinant
of which is equal to $- \omega(\lambda)$  therefore is not vanish by
assumption. Solving that system we find
\begin{eqnarray}\label{2.12}
c_{11}=\frac{1}{\omega_{2}(\lambda)}\int_{c}^{b}\chi_{2}(y,\lambda)F_{1}(y)dy+\frac{F_{2}}{\omega_{2}(\lambda)},
\end{eqnarray}
\begin{eqnarray}\label{2.13}
c_{22}=\frac{1}{\omega_{1}(\lambda)}\int_{a}^{c}\phi_{1}(y,\lambda)F_{1}(y)dy.
\end{eqnarray}
Putting (\ref{2.11}), (\ref{2.12}) and (\ref{2.13}) in (\ref{2.10})
we deduce that problem (\ref{2.5})-(\ref{2.89}) has an unique
solution,
\begin{eqnarray}\label{2.14}
u(x,\lambda)=\left\{\begin{array}{c}
               \frac{\chi_{1}(x,\lambda)}{\omega_{1}(\lambda)}\int_{a}^{x}\phi_{1}(y,\lambda)F_{1}(y)dy +
\frac{\phi_{1}(x,\lambda)}{\omega_{1}(\lambda)}(\int_{x}^{c}\chi_{1}(y,\lambda)F_{1}(y)dy \\ +\frac{\delta_{1}\delta_{2}}{\gamma_{1}\gamma_{2}}\int_{c}^{b}\chi_{2}(y,\lambda)F_{1}(y)dy+\frac{\delta_{1}\delta_{2}}{\gamma_{1}\gamma_{2}}F_{2}) \ \ \ \ \ \ \ \ \ for \  x \in (a,c) \\
                \\
\frac{\chi_{2}(x,\lambda)}{\omega_{2}(\lambda)}(\frac{\gamma_{1}\gamma_{2}}{\delta_{1}\delta_{2}}\int_{a}^{c}\phi_{1}(y,\lambda)F_{1}(y)dy+\int_{c}^{x}\phi_{2}(y,\lambda)F_{1}(y)dy)\\
+
\frac{\phi_{2}(x,\lambda)}{\omega_{2}(\lambda)}(\int_{x}^{b}\chi_{2}(y,\lambda)F_{1}(y)dy+F_{2})\  \ \ \ \ \ \ \ \ for \  x \in (c,b) \\
             \end{array}\right.
\end{eqnarray}
By defining the Green' s function as
\begin{eqnarray}\label{2.15}
G_{1}(x,y;\lambda)=\left\{\begin{array}{c}
 \frac{\phi(x,\lambda) \chi(y,\lambda)}{\omega(\lambda)} \ \ \ \ \textrm{for} \  a\leq y\leq x\leq b, \ \  x, y \neq c\\
                \\
  \frac{\phi(y,\lambda)\chi(x,\lambda)}{\omega(\lambda)} \ \ \ \ \textrm{for} \  a\leq x\leq y\leq b, \ \  x, y\neq c \\
             \end{array}\right.
\end{eqnarray}
the formula (\ref{2.14}) can be rewritten in the following form
\begin{eqnarray}\label{2.16}
u(x,\lambda)&=&\gamma_{1}\gamma_{2} \int_{a}^{c}
G_{1}(x,y;\lambda)F_{1}(y)dy+ \delta_{1}\delta_{2} \int_{c}^{b}
G_{1}(x,y;\lambda)F_{1}(y)dy \nonumber\\&+&
\delta_{1}\delta_{2}F_{2}\frac{\phi(x,\lambda)}{\omega(\lambda)}.
\end{eqnarray}
\section{Operator formulation in modified Hilbert space }
In this section we shall introduce a special  equivalent inner
product in the Hilbert space $L_{2}[a,b]\oplus \mathbb{C}$   and
define a symmetric operator A in this space such a way that the
considered problem can be interpreted as the eigenvalue problem of
this operator. For this we assume that,
$\rho:=\beta'_{1}\beta_{2}-\beta_{1}\beta'_{2}>0$ and for the sake
of shorting we restrict ourselves to the investigation
only the case $\gamma_{i}\neq 0, \delta_{i}\neq 0 (i=1,2)$ .\\
 In the Hilbert Space $H=L_{2}[a,b]\oplus \mathbb{C}$  of two-component vectors we define an equivalent inner product by
$$
<F,G>_{1}:=|\gamma_{1}\gamma_{2}| \ \int_{a}^{c}
F_{1}(x)\overline{G_{1}(x)}dx + |\delta_{1}\delta_{2}| \int_{c}^{b}
F_{1}(x)\overline{G_{1}(x)}dx +
\frac{|\delta_{1}\delta_{2}|}{\rho}F_{2}\overline{G_{2}}
$$
for $F=\left(
  \begin{array}{c}
   F_{1}(x),
    F_{2} \\
  \end{array}
\right)$,\quad $G=\left(
  \begin{array}{c}
    G_{1}(x),
  G_{2} \\
  \end{array}
\right)\in H$ and apply operator theory in the modified Hilbert
space $H_{1}=(L_{2}[a,b]\oplus\mathbb{C},<.,.>_{1} ).$ Below we
shall use the following notations:
$$(u)_{\beta}:= \beta_{1}u(b)-\beta_{2}u'(b),  \ \  (u)'_{\beta}:= \beta'_{1}u(b)-\beta'_{2}u'(b).$$
Let us define a linear operator $H:A\rightarrow A$   with the domain

\begin{eqnarray*}\label{2.3}
D(A):=& \bigg \{&F=\left(
  \begin{array}{c}
    F_{1}(x),
    (F_{1})'_{\beta} \\
  \end{array}
\right):F_{1}(x) \ \textrm{and} \  F'_{1}(x) \ \textrm{are absolutely} \nonumber \\
&{}&\textrm{continuous in each interval [a,c) \ \textrm{and} (c,b]},
\textrm{and has a finite limits}\nonumber \\
&{}&  F_{1}(c\mp0) \ \textrm{and} \ F'_{1}(c\mp0), \ \tau F_{1} \in
L_{2}[a,b], \  a_{1}u(a)+a_{2}u'(a)=0,\nonumber \\
&{}& \gamma _{1}F_{1}(c-0)=\delta_{1} F_{1}(c+0), \ \gamma
_{2}F'_{1}(c-0)=\delta_{2} F'_{1}(c+0)\bigg \}
\end{eqnarray*}
and action low\\

$A(F_{1}(x), (F_{1})'_{\beta})=(\tau F_{1}, (-F_{1})_{\beta}).$\\ \\
Consequently the problem $(\ref{1.1})-(\ref{1.5})$ can be written in
the operator form as $$AU=\lambda U, \ \ U:=(u(x), (u)'_{\beta})
 \in D(A)$$
in the Hilbert space $H_{1}$.

\begin{lem}\label{lem3.1}
The domain  $D(A)$ is dense in $H_{1}$.
\end{lem}
\begin{proof}
\end{proof}
\begin{thm}\label{2.1}
If \ $ \gamma_{1} \gamma_{2} \delta_{1} \delta_{2}>0$ then the
linear operator $A$  is symmetric.
\end{thm}
\begin{proof}

\end{proof}

\begin{rem}\label{rem1}
Having in view this property of the problem
$(\ref{1.1})-(\ref{1.5})$, we shall assume $ \gamma_{1} \gamma_{2}
\delta_{1} \delta_{2}>0$  everywhere in below. Also without loss of
generality we shall let $ \gamma_{1} \gamma_{2}>0$ \textrm{and} \ $
\delta_{1} \delta_{2}>0$ .
\end{rem}

\begin{rem}\label{rem2}By Lemma \ref{2.1} all eigenvalues of the problem $(\ref{1.1})-(\ref{1.5})$ are real.
Consequently, we can now assume that all eigenfunctions are
real-valued.
\end{rem}

\begin{cor}\label{cor2}Let $u(x)$ \textrm{and} $v(x)$   be eigenfunctions corresponding to distinct eigenvalues. Then
\begin{eqnarray}\label{2.3}
\gamma_{1}\gamma_{2} \ \int_{a}^{c} u(x)v(x)dx +
\delta_{1}\delta_{2} \int_{c}^{b} u(x)v(x)dx +
\frac{\delta_{1}\delta_{2}}{\rho}(u)'_{\beta}(v)'_{\beta}=0.
\end{eqnarray}
\end{cor}
\begin{proof}
The proof is immediate from the fact that, the eigenelements\\
$(u(x), (u)_{\beta}^{'}) \ \textrm{and} \ (v(x), (v)_{\beta}^{'})$
of the symmetric linear operator $A$ is ortogonal  in the Hilbert
space $H_{1}.$
\end{proof}
\section{ The Resolvent operator and Self-adjointness of the problem  }
In this section we shall construct the  Resolvent operator and prove
self-adjointness of the problem.
\begin{lem}\label{bag}
Let $\lambda_{0}$  be zero of $w(\lambda)$. Then the solutions
$\phi(x,\lambda_{0}) \ \textrm{and} \ \chi(x,\lambda_{0})$ are
linearly dependent.
\end{lem}
\begin{proof}
\end{proof}
\begin{thm}\label{thm1}
Each eigenvalue of the problem (\ref{1.1})-(\ref{1.5}) is the simple
zero of \ $w(\lambda)$.
\end{thm}
\begin{proof}
\end{proof}
Let $A$ be defined as above and let $\lambda$  not be an eigvalue of
this operator. For construction the resolvent operator $R(\lambda,
A):=(\lambda-A)^{-1}$ we shall solve the operator equation
\begin{eqnarray}\label{2.4}
(\lambda-A)U=F
\end{eqnarray}
for $F=(F_{1}(x),
  F_{2}) \in  H_{1}$. This operator equation equivalent to the problem
(\ref{2.5})-(\ref{2.89}).\\ Using the equalities  we see that
\begin{eqnarray}\label{2.17}
(G_{1}(x,.;\lambda))'_{\beta}=\frac{\phi(x,\lambda)}{\omega(\lambda)}(\chi(x,\lambda))'_{\beta}=\rho\frac{\phi(x,\lambda)}{\omega(\lambda)}.
\end{eqnarray}
Hence, \ the solution (\ref{2.14}) may be  written as
\begin{eqnarray}\label{2.18}
u(x,\lambda)&=&\gamma_{1}\gamma_{2} \int_{a}^{c}
G_{1}(x,y;\lambda)F_{1}(y)dy+ \delta_{1}\delta_{2} \int_{c}^{b}
G_{1}(x,y;\lambda)F_{1}(y)dy \nonumber\\&+&
\frac{\delta_{1}\delta_{2}}{\rho}(G_{1}(x,.;\lambda))'_{\beta}F_{2}
\end{eqnarray}
Consequently, for the solution
\begin{eqnarray}\label{2.19}
U(F,\lambda):=(u(x,\lambda), (u(.,\lambda))'_{\beta})
\end{eqnarray}
of the nonhomogeneous operator equation (\ref{2.4}) we obtain the
following formula
\begin{eqnarray}\label{2.20}
U(F,\lambda):=(<G_{x,\lambda},\overline{F}>_{1},
(<G_{x,\lambda},\overline{F}>_{1})'_{\beta})
\end{eqnarray}
where
\begin{eqnarray}\label{2.21}
G_{x,\lambda}:=(G_{1}({x,.;\lambda}),
(G_{1}({x,.;\lambda}))'_{\beta})
\end{eqnarray}
Now, making use (\ref{2.15}), (\ref{2.18}), (\ref{2.19}),
(\ref{2.20}) and (\ref{2.21}) we see that if $\lambda$  not an
eigenvalue of A then
\begin{eqnarray}\label{2.22}
U(F,\lambda)\in D(A) \ \textrm{for} \ \  F \in H_{1},
\end{eqnarray}
\begin{eqnarray}\label{2.23}
U((\lambda-A)F,\lambda)=F, \ \textrm{for} \in D(A)
\end{eqnarray}
and
\begin{equation}\label{2.24} \| U(F,\lambda) \| \leq
\ |Im\lambda|^{-1}\| F \| \ \textrm{for} \ F \in H_{1}, \ \ \ Im
\lambda \neq0 .
\end{equation}
Hence, each nonreal $\lambda\in \mathbb{C}$  is a regular point of
an operator A and
\begin{eqnarray}\label{2.25}
R(\lambda,A)F=(<G_{x,\lambda},\overline{F}>_{1},
(<G_{x,\lambda},\overline{F}>_{1})'_{\beta}) \ \textrm{for} \ \  F
\in H_{1}
\end{eqnarray}
Because of (\ref{2.22}) and (\ref{2.25})
\begin{eqnarray}\label{2.26}
(\lambda-A)D(A)=(\overline{\lambda}-A)D(A)=H_{1} \ \textrm{for} \
Im\lambda \neq0.
\end{eqnarray}
\begin{thm}\label{es}
The linear operator $A$ is self-adjoint.
\end{thm}
\begin{proof}
 From the equality (\ref{2.26}) and the fact that $A$ is symmetric it
follows by the standard theorems for symmetric operators in Hilbert
spaces that $A$ is self-adjoint  in $H_{1}$ (see, for example,
\cite{la}, Theorem 2.2.p. 198). \end{proof}

\section{Expansion is series of eigenfunctions }

Let   $\lambda_{n}, \ n=1,2,...$ are eigenvalues of the operator A
and let $\phi_{n}(x):=\phi(x,\lambda_{n}), n=1,2,...$ be defined as
in section 2. By virtue of Lemma \ref{bag} the two-component vectors
\begin{eqnarray}\label{3.1}
\Phi_{n}:=(\phi(x,\lambda_{n}), (\phi(.,\lambda_{n}))'_{\beta}), \ \
n=0,1,2,...
\end{eqnarray}
are the eigenelements of $A$. Moreover,
\begin{eqnarray}\label{3.2}
<\Phi_{n},\Phi_{m}>_{1}=0 \ \ \ for \  n\neq m
\end{eqnarray}
since $A$ is self-adjoint in $H_{1}$. Denote the normalized
eigenelements by
\begin{eqnarray}\label{3.3}
\Psi_{n}:=(\psi_{n}(x), (\psi_{n}(x))'_{\beta}),
\end{eqnarray}
where
\begin{eqnarray}\label{3.4}
\psi_{n}(x):=\frac{\phi(x,\lambda_{n})}{ \| \Phi_{n} \|_{1} }.
\end{eqnarray} Let $k_{n}\neq0$   denote the real constant for which
\begin{eqnarray}\label{3.5}
\chi(x,\lambda_{n})=k_{n}\phi(x,\lambda_{n}), \ \ n=0,1,2,...   \
x\in(a,c)\cup(c,b).
\end{eqnarray} Then
 \begin{eqnarray}\label{3.6}
 (\phi_{n}(x))'_{\beta}=\frac{\rho}{k_{n}}.
\end{eqnarray}
Writing  for $\lambda_{n}$   instead of $\lambda_{0}$  we obtain
\begin{eqnarray}\label{3.7}
  \| \phi_{n} \|_{1}^{2}=\frac{\omega'(\lambda_{n})}{k_{n}}.
\end{eqnarray}
Now, making use the representation (\ref{2.18}) of the solution
$u(x,\lambda)$, the equalities (\ref{2.15}), (\ref{3.3})-
(\ref{3.5}) and the fact that each eigenvalue $\lambda_{n}$  is
simple zero of $\omega(\lambda)$  we derive that
\begin{eqnarray}\label{3.8}
Res_{\lambda=\lambda_{n}}u(x,\lambda)=<F,\Psi_{n}>_{1}\psi_{n}(x).
\end{eqnarray}
Consequently,
\begin{eqnarray}\label{3.9}
Res_{\lambda=\lambda_{n}}R(\lambda,A)F:=<F,\Psi_{n}>_{1}\Psi_{n}=C_{n}(F)\Psi,
\end{eqnarray}
where
\begin{eqnarray}\label{3.10}
C_{n}(F):=<F,\Psi_{n}>_{1}
\end{eqnarray}
are Fourier coefficients.
\begin{thm}\label{th3.1}(i) The modified Parseval equality
\begin{eqnarray}\label{a1}
\gamma_{1}\gamma_{2}  \ \int_{a}^{c} f^{2}(x)dx +
\delta_{1}\delta_{2} \int_{c}^{b}
f^{2}(x)dx&=&\sum_{n=0}^{\infty}\mid \gamma_{1}\gamma_{2}  \
\int_{a}^{c} f(x)\psi_{n}(x)dx \nonumber\\ &+& \delta_{1}\delta_{2}
\int_{c}^{b}f(x)\psi_{n}(x)dx\mid^{2}
\end{eqnarray}
is  hold for each $f \in L_{2}[a,c] \oplus L_{2}[c,b].$

\end{thm}

\begin{proof}
\end{proof}

\begin{thm}\label{thm3.2} Let $(f(x), (f)_{\beta}^{'}) \in
D(A)$. Then
\begin{eqnarray}\label{3.16} (i) \ \
 f(x)&=& \sum_{n=0}^{\infty}\big(\gamma_{1}\gamma_{2}  \
\int_{a}^{c} f(x)\psi_{n}(x)dx  + \delta_{1}\delta_{2}
\int_{c}^{b}f(x)\psi_{n}(x)dx\nonumber\\
&+&\frac{\delta_{1}\delta_{2}}{\rho}(f)_{\beta}^{'}
(\psi_{n})_{\beta}^{'} \big)\psi_{n}(x) \end{eqnarray} where, the
series converges absolutely and uniformly in whole $[a,c)\cup(c,b].$
(ii) The series (\ref{3.16}) may also be differentiated, the
differentiated series also being absolutely and uniformly convergent
in whole $[a,c)\cup(c,b].$
 \end{thm}
 \begin{proof}
\end{proof}
 \section{Counterexample}
Recall that we had derived all results in this study under condition
$ \gamma_{1} \gamma_{2} \delta_{1} \delta_{2}>0$. Let us show that
this simple condition on the sign of the coefficients of the
transmission conditions can not be omitted without putting any other
condition on this coefficients. For this, consider the following
special case of the problem (\ref{1.1})-(\ref{1.5}) for which $
\gamma_{1} \gamma_{2} \delta_{1} \delta_{2}<0$ :
\begin{equation}\label{6.1}
\ -u''=\lambda u, \ \ \ \ \ x \in [-1,0)\cup (0,1]
\end{equation}
\begin{equation}
 \ \ u(-1)=0 \ \ , \
\ \lambda u(1)=u'(1)
\end{equation}
\begin{equation}
 \ \ u(0-0)=u(0+0) \ \ , \
\ \ u'(0-0)=-u'(0+0)
\end{equation}
It is easy to verify that this problem has only the trivial solution
\\$u=0 \  \textrm{for any } \lambda \in \mathbb{C}$. Thus, if $
\gamma_{1} \gamma_{2} \delta_{1} \delta_{2}<0$ then the spectrum of
the problem (\ref{1.1})-(\ref{1.5}) may be empty.

\end{document}